\documentclass[12pt,letterpaper, reqno]{amsart}
\usepackage[T1]{fontenc}
\usepackage[latin9]{inputenc}
\usepackage{amsmath, amsthm, amssymb}
\usepackage{mathrsfs}

\usepackage{xcolor}
\usepackage{dsfont}

\usepackage{cancel}
\usepackage{soul}
\usepackage{comment}
\usepackage{color}
\usepackage{amscd}
\usepackage{tabularx}
\usepackage{url}
\usepackage{eurosym}
\usepackage{braket}
\usepackage{hyperref}
\usepackage{enumerate}
\usepackage{graphicx}

\makeatletter
\def\widebreve{\mathpalette\wide@breve}
\def\wide@breve#1#2{\sbox\mathbb{Z}@{$#1#2$}%
     \mathop{\vbox{\m@th\ialign{##\crcr
\kern0.08em\brevefill#1{0.8\wd\mathbb{Z}@}\crcr\noalign{\nointerlineskip}%
                    $\hss#1#2\hss$\crcr}}}\limits}
\def\brevefill#1#2{$\m@th\sbox\tw@{$#1($}%
  \hss\resizebox{#2}{\wd\tw@}{\rotatebox[origin=c]{90}{\upshape(}}\hss$}
\makeatletter

\makeatletter
\numberwithin{equation}{section}
\numberwithin{figure}{section}
\theoremstyle{plain}
\newtheorem{thm}{\protect\theoremname}[section]
\theoremstyle{plain}

\ifx\proof\undefined
\newenvironment{proof}[1][\protect\proofname]{\par
	\normalfont\topsep6\p@\@plus6\p@\relax
	\trivlist
	\itemindent\parindent
	\item[\hskip\labelsep\scshape #1]\ignorespaces
}{%
	\endtrivlist\@endpefalse
}
\providecommand{\proofname}{Proof}
\fi
\theoremstyle{remark}

\theoremstyle{plain}
\newtheorem{lem}[thm]{\protect\lemmaname}
\usepackage{dsfont}

\newcommand\1{\mathds{1}}

\makeatother

\usepackage{babel}
\providecommand{\corollaryname}{Corollary}
\providecommand{\lemmaname}{Lemma}
\providecommand{\remarkname}{Remark}
\providecommand{\theoremname}{Theorem}

\numberwithin{equation}{section}
\numberwithin{figure}{section}
\theoremstyle{plain}

\newtheorem{prop}[thm]{\protect Proposition}

\newtheorem{theorem}[thm]{Theorem}
\newtheorem{lemma}[thm]{Lemma}
\newtheorem{corollary}[thm]{Corollary}

\newtheorem{exmp}[thm]{Example}

\def\d1{{\textcolor{red} {d-1}}}

\def \bC {\mathbb C}

\def \bF {\mathbb F}

\def \bI {\mathbb I}

\def \bN {\mathbb N}

\def \bR {\mathbb R}
\def \bZ {\mathbb Z}

\def \bS {\mathbb S}

\def \bh {\mathbf h}

\def \bl {\boldsymbol \ell}

\def \bn {\mathbf n}

\def \br {\mathbf r}

\def \bv {\mathbf v}
\def \bx {\mathbf x}

\def \bzero {{\boldsymbol 0}}

\def \balp {{\boldsymbol{\alp}}}

\def \bxi {{\boldsymbol{\xi}}}

\def \fy {\mathfrak y}
\def \fz {\mathfrak z}

\def \le {\leqslant}
\def \leq {\leqslant}
\def \ge {\geqslant}
\def \geq {\geqslant}

\def \Tr {\mathrm{Tr}}

\def \d {{\mathrm{d}}}

\def \ds1 {\mathds{1}}

\def \alp {{\alpha}}

\def \del {{\delta}}
\def \eps {{\varepsilon}}
\def \kap {{\kappa}}
\def \lam {{\lambda}}

\newcommand{\Fr}[1]{\widehat{#1}}
\newcommand{\n}[1]{\|{#1}\|}

\makeatother

\usepackage{babel}
\providecommand{\lemmaname}{Lemma}

\providecommand{\theoremname}{Theorem}
\usepackage{xifthen,ifthen}
\makeatletter
\newcounter{@ToDo}
\newcommand{\todo@helper}[1]{%
	({\color{blue}TODO~\arabic{@ToDo}: {#1\@addpunct{.}}})%
}
\newcommand{\todo}[1]{\stepcounter{@ToDo}%
	\relax\ifmmode\text{\todo@helper{#1}}%
	\else\todo@helper{#1}\fi%
}
\newcounter{@cdo}
\newcommand{\cdo@helper}[1]{%
	({\color{red}CITE~\arabic{@cdo}: {#1\@addpunct{.}}})%
}
\newcommand{\cdo}[1]{\stepcounter{@cdo}%
	\relax\ifmmode\text{\cdo@helper{#1}}%
	\else\cdo@helper{#1}\fi%
}
\newcommand{\mmod}[1]{\,\,\mathrm{mod}\,\,#1}

\def \ds1 {\mathds{1}}

\def \alp {{\alpha}}

\def \del {{\delta}}
\def \eps {{\varepsilon}}
\def \epsilon {{\varepsilon}}
\def \kap {{\kappa}}
\def \lam {{\lambda}}

\usepackage{ragged2e}

\begin{document}

\author{Sam Chow \and Zi Li Lim \and Akshat Mudgal}
\address{Mathematics Institute, Zeeman Building, University of Warwick, Coventry CV4 7AL, United Kingdom}
\email{sam.chow@warwick.ac.uk}
\address{Department of Mathematics, UCLA, Los Angeles, CA 90095, USA}
\email{zililim@math.ucla.edu}
\address{Mathematics Institute, Zeeman Building, University of Warwick, Coventry CV4 7AL, United Kingdom}
\email{akshat.mudgal@warwick.ac.uk}

\title[Dense Fermat over finite fields and rings]{Generalised Fermat equations in dense variables over finite fields and rings}

\subjclass[2020]{11B30 (primary); 11D41, 11D45, 11P05, 11T23 (secondary)}

\keywords{Arithmetic Ramsey theory, generalised Fermat equations, Bohr sets, character sums}

\begin{abstract}
Let $A$ be a sufficiently dense subset of a finite field $\mathbb F_q$ or a finite, cyclic ring $\bZ/N\bZ$. Assuming that $q$ and $N$ have no small prime divisors, we show that generalised Fermat equations have the expected number of solutions over $A$. We further show that our density threshold is optimal. 
Our proofs involve average Fourier decay for Bohr sets, mixed character sum bounds, equidistribution of polynomial sequences, popular Cauchy--Davenport lemmas, and a regularity-type lemma due to Semchankau. 
\end{abstract}

\maketitle

{\begin{center}
\footnotesize \em{Dedicated to Trevor Wooley on the occasion of his $(5^3 - 4^3 - 1^3)^{\mathrm{th}}$ birthday}
\end{center}}

\section{Introduction}

This paper studies solutions to Fermat equations over dense subsets of finite fields and cyclic rings. Such a problem seems to have been first studied by  Csikv\'{a}ri--Gyarmati--S\'{a}rk\"{o}zy \cite{CGS2012}, who considered solutions to the equation
\begin{equation} \label{sumsquare}
x + y = z^2
\end{equation}
with the variables lying in dense subsets of finite fields. They observed that the set 
\begin{equation} 
\label{densecount}
A = \{ x \in \mathbb{F}_p : x \in (0, p/4)  \ \text{and} \ x^2 \in (p/2,p) \},
\end{equation}
while satisfying $|A| \gg p$, does not exhibit any solutions to \eqref{sumsquare}. They further asked whether a colouring version of this result might exist, that is, for any partition $\mathbb{F}_p = C_1 \cup \dots \cup C_r$ with $r \in \mathbb{N}$, does there always exist some $1 \leq i \leq r$ such that $C_i$ has solutions to \eqref{sumsquare}? The latter question, along with its generalisation to more broader equations of the shape 
\begin{equation}
\label{gensumsquare}
x^{\alpha} +  y^{\beta} = z^{\gamma},
\end{equation} 
was resolved by Lindqvist \cite{Lin2018}. In fact, Lindqvist proved that for any $\alpha, \beta, \gamma \in \mathbb{N}$ and any partition $\mathbb{F}_p = C_1 \cup \dots \cup C_r$, with $r \in \mathbb{N}$,  there must exist some $1 \leq i \leq r$ such that $C_i$ has $\gg_{r, \alpha, \beta, \gamma} p^2$ many solutions to \eqref{gensumsquare}. This matches the total number of solutions to \eqref{gensumsquare} in $\mathbb{F}_p$, up to multiplicative constants. 
Such a phenomenon has been referred to as \emph{supersaturation},
by analogy with the extremal combinatorics parlance. We refer the reader to the pioneering works of Varnavides \cite{Va1959} and Frankl--Graham--R\"{o}dl~\cite{FGR1988}, who established these types of results for linear equations over integers. See also \cite{CC2025, CLP2021, Pre2021} and \cite{Lim2025} for supersaturation results for non-linear equations over the integers and finite, cyclic rings, respectively.

Returning to the set $A$ in \eqref{densecount}, one can prove using standard Fourier-analytic methods that $|A| \geq p/8 - o(p)$. A natural question then is whether there is any larger set that does not contain solutions to \eqref{sumsquare}. A straightforward consequence of the results in this paper is that this is roughly the best possible construction.

\begin{corollary} 
\label{denselin}
For every $\kappa >0$, every prime $p$ sufficiently large in terms of $\kappa$, and every set $A \subseteq \mathbb{F}_p$ satisfying $|A| \geq p(1/8 + \kappa)$, one has 
\[ \sum_{x,y, z \in A} \1_{x+y = z^2} \gg_{\kappa} p^2. \]
\end{corollary}

In fact, our results in this paper provide optimal density thresholds along with supersaturation for a wider family of equations in broader algebraic settings. Writing $[n] = \{1,2,\dots, n\}$ for every $n \in \mathbb{N}$, our first main result over general finite fields is as follows.







\begin{thm} 
\label{thm1}
Let $r$ and $i_1 < \dots < i_r$ be positive integers, and let $s \ge 3$ be an integer. Let
$$
[s] = I_1 \cup \dots \cup I_r 
$$
be a partition of $[s]$ into discrete intervals $I_1, \dots, I_r$ such that $|I_j| = k_j$ for all $j$. Let $\kap > 0$, and let $\eps > 0$ be sufficiently small in terms of $s, i_r, \kappa$. Let $p$ be a prime sufficiently large in terms of $s, i_r, \kappa$, and let $q = p^m$ for some $m \in \mathbb{N}$. Let
$u \in \mathbb{F}_q$, let $c_1,\ldots,c_s \in \bF_q \setminus \{ 0 \}$, and let $A \subseteq \mathbb{F}_q$ with
\begin{equation}
\label{DensityAssumption}
|A| \geq \left(  \prod_{1 \le j \le r} \frac1{rk_j} + \kappa \right)q.
\end{equation}
Then there are at least $\eps q^{s-1}$ many solutions $\bx \in A^s$ to 
\begin{equation}
\label{MainEq}
\sum_{1 \leq j \leq r} \sum_{n \in I_j} c_n x_n^{i_j} = u.
\end{equation}
\end{thm}

\noindent
Corollary \ref{denselin} can be deduced directly from Theorem \ref{thm1} by setting 
\[
r=2,
\quad
(i_1, i_2) = (1,2),
\quad
[3] = \{1,2\} \cup \{3\},
\quad
u = 0.
\]

Another setting of interest is the case when $r=1$ and $i_1 = k$. This can be interpreted as a dense Waring-type problem in $\mathbb{F}_p$.

\begin{corollary} 
Let $s \geq 3$ and $k \ge 1$ be integers. Let $\kap > 0$, and let $\eps >0$ be sufficiently small in terms of $s,k, \kap$. Let $p$ be a prime sufficiently large in terms of $s,k, \kappa$, and let $q = p^m$ for some $m \in \mathbb{N}$. Let $u \in \bF_q$, let $c_1, \ldots, c_s \in \bF_q \setminus \{ 0 \}$, and let $A \subseteq \mathbb{F}_q$ with
\[ \frac{|A|}{q} \geq \frac{1}{s} + \kappa. \]
Then there are at least $\eps q^{s-1}$ many solutions $\bx \in A^s$ to 
\begin{equation}
\label{WaringEq}
c_1 x_1^k + \cdots + c_s x_s^k = u.
\end{equation}
\end{corollary}

Waring's problem over $\bN$ concerns representing all positive integers --- or all but finitely many --- as sums of at most $s$ many $k^{\mathrm{th}}$ powers of positive integers. The majority of interest has revolved around the search for the best possible dependence of $s$ on $k$, see \cite{BW2023}. More recently, there has been work which has focused on proving dense versions of Waring's problem, that is, given a sufficiently dense subset $\mathcal{A}$ of the $k^{\mathrm{th}}$ powers, one wants to represent all sufficiently large positive integers as a sum of at most $s$ many elements from $\mathcal{A}$, see \cite{Sal2021}. One can also replace the $k^{\textrm{th}}$ powers with other arithmetically interesting sets, see the very nice work of Shao \cite{Sha2014} which considers the analogous problem over dense subsets of primes, see also \cite{LP2010}. All the aforementioned works proceed via applying the transference principle from additive combinatorics \cite{Gr2005} to switch this problem to the $\mathbb{Z} /N \mathbb{Z}$ setting, for some appropriately chosen $N \in \mathbb{N}$. 

Our second main result is an optimal version of Theorem \ref{thm1} over finite, cyclic rings $\mathbb{Z}/N\mathbb{Z}$, for suitably rough $N$.

\begin{thm}  
\label{thm2}
Let $r$ and $i_1 < \dots < i_r$ be positive integers, and let $s \ge 3$ be an integer. Let
$
[s] = I_1 \cup \dots \cup I_r 
$
be a partition of $[s]$ into discrete intervals $I_1, \dots, I_r$ such that $|I_j| = k_j$ for all $j$. Let $\kap > 0$, let $\Omega \geq 1$, and let $\eps > 0$ be sufficiently small in terms of $s, i_r, \kappa, \Omega$. Then there exists a large constant $R$ depending only on $s, i_r, \kappa, \Omega$ such that the following is true.

Let $N=p_1^{m_1} \cdots p_t^{m_t}$ with 
\begin{equation}
\label{rough}
t, m_1,\dots, m_t \leq \Omega \quad \text{and} \quad p_1, \dots, p_m \geq R.
\end{equation}
Let
$u \in \bZ/N\bZ$, let $c_1,\ldots,c_s \in (\bZ/N\bZ)^{\times}$, and let $A \subseteq \bZ/N\bZ$ with
\begin{equation}
\label{density assumption mod N}
|A| \geq \left(  \prod_{1 \le j \le r} \frac1{rk_j} + \kappa \right) N.
\end{equation}
Then there are at least $\eps N^{s-1}$ many solutions $\bx \in A^s$ to \eqref{MainEq}.
\end{thm}

\noindent
Above and in the sequel, the $p_j$ are understood to be pairwise distinct primes.

As before, Theorem \ref{thm2} dispenses the following optimal dense Waring result in $\mathbb{Z}/N\mathbb{Z}$ for suitably rough numbers $N$.

\begin{corollary}  \label{warcor}
Let $s \geq 3$ and $k \ge 1$ be integers. Let $\kap > 0,$ let $ \Omega \geq 1$, and let $\eps >0$ be sufficiently small in terms of $s,k, \kap, \Omega$. Then, there exists a large constant $R$ depending only on $s, k, \kappa, \Omega$ such that the following is true.

Let $N=p_1^{m_1} \cdots p_t^{m_t} \in \bN$ with \eqref{rough}.
Let $c_1,\ldots,c_s \in (\bZ/N\bZ)^{\times}$, let
$u \in \bZ/N\bZ$, and let $A \subseteq \bZ/N\bZ$ with
\[ \frac{|A|}{N} \geq \frac{1}{s} + \kappa. \]
Then there are at least $\eps N^{s-1}$ many solutions $\bx \in A^s$ to \eqref{WaringEq}.
\end{corollary}

The only previously known result in this direction is due to the second author \cite{Lim2025}, who considered the setting of Corollary \ref{warcor} when $k =2$ and $s \geq 5$, but with more general choices of $N$. Hence, in comparison, our method gives optimal density results for higher powers with potentially distinct exponents in much fewer variables, although we are required to restrict to rough numbers $N$ with few prime factors, while the techniques of \cite{Lim2025} work for more general rough numbers $N$. In fact, the approach in \cite{Lim2025} is quite different from our current paper. The former relies on restriction estimates combined with the transference principle. We deploy regularity-type arguments of Semchankau~\cite{Sem2025}, together with a popular Cauchy--Davenport lemma and results on equidistribution of polynomial sequences in Bohr sets, see Proposition~\ref{equidistribution mod N}.

We emphasise that Theorems \ref{thm1} and \ref{thm2} deliver optimal density thresholds. In order to see this quickly, we consider the following example over $\mathbb{F}_p$.

\begin{exmp}
Given $\lambda_1, \dots, \lambda_r \in \bF_p \setminus \{0\}$, we set $c_n = \lambda_j$ for every $1 \leq j \leq r$ and every $n \in I_j$, and let
\[ A  = \bigg\{ x \in \bF_p : \lam_j x^{i_j} \in \Big( 0, \frac{p}{k_j r} \Big)  \ ( 1\leq j \leq r) \bigg\}. \]
One may apply Lemma  \ref{writ}, along with the fact that $\n{\Fr{\mathds{1}_{P}}}_1 \ll \log p$ for any arithmetic progression $P$ in $\bF_p$, to deduce that
\[ |A| = (1 + o(1)) p \prod_{1 \leq j \leq r} \frac{1}{k_j r}. \]
Moreover $A$ has no solutions to \eqref{MainEq} when $u = 0$, since for every $1 \leq j \leq r$ and every $\{a_n\}_{n \in I_j} \subseteq A$, one has
\[ \lambda_j \sum_{n \in I_j}  a_n^{i_j} \in  \Big( 0, \frac{p}{r} \Big) , \quad \text{whence} 
\quad \sum_{j=1}^r \sum_{n \in I_j} c_n a_n^{i_j} \in (0,p). \]
\end{exmp}

Furthermore, we note that it is necessary to have $N$ be at least somewhat rough in Theorem \ref{thm2} and Corollary \ref{warcor}, as 
evinced by the following example.

\begin{exmp}
Let $k\in \bN$, let $N=3N'$ where $N'$ is coprime to $3$, and let
\begin{equation*}
    A=\{x\in \bZ/N\bZ:x^k\equiv 1\text{ $\mathrm{mod}$ $3$}\}.
\end{equation*}
Then $A$ has large density, which is $1/3$ when $k$ is odd and $2/3$ when $k$ is even. However, owing to the congruence obstruction, it is impossible to represent all elements in $\bZ/N\bZ$ by $x_1^k + \cdots + x_s^k$ with $x_j\in A$ for all $1\leq j\leq s$, no matter how large $s$ is.
\end{exmp}



It would be interesting to study the problem over function fields. Equidistribution of polynomial sequences --- a key ingredient in our approach --- was recently established in \cite{LLW2025} and \cite{CGLLW}.

\bigskip

Let us finish our introduction with a brief description of the proof strategy.
To fix ideas, consider \eqref{WaringEq}. Assume for a contradiction that the iterated sumset $A_1 + \cdots + A_s$ does not contain $u$ with high multiplicity, where 
$
A_i = \{ c_i a^k: a \in A \}
.
$
Semchankau's wrapping lemma tells us that each $A_i$ is mostly contained in a \emph{wrapper} $W_i$, where a wrapper is a union of a small number of inhomogeneous, low-rank Bohr sets (it is important that we control the parameters). Moreover, it tells us that $W_1 + \dots + W_s$ does not fill the group with high multiplicity. Combining this with Green and Ruzsa's popular Cauchy--Davenport lemma allows us to deduce that the average density of $W_1, \dots, W_s$ is at most $1/s + o(1)$. Next, we establish strong average Fourier decay estimates for Bohr sets and hence for wrappers.  We infer from this --- and an extra slicing argument for moduli that are not prime powers --- effective equidistribution of $k^{\mathrm{th}}$ powers in a wrapper. This, along with the upper bound on the average density of the wrappers, gives us an upper bound on the density of $A$, delivering the desired contradiction.

\bigskip

\subsection*{Organisation} 
We record some preliminary definitions and results in \S2, including a popular version of the Cauchy--Davenport lemma along with various Weil-type bounds on exponential sums. In \S3, we record a variation of Semchankau's so-called Wrapping Lemma \cite{Sem2025}. We prove Theorem \ref{thm1} in \S4, and we prove some equidistribution results along with Theorem \ref{thm2} in \S5. 

\subsection*{Notation} We use Vinogradov notation, that is, we write $X \ll Y$ to mean there exists some absolute constant $C>0$ such that $|X| \leq C Y$. We denote $X \ll_z Y$ to mean that the above implicit constant $C$ may depend on $z$. We write $X = O_z(Y)$ to mean $X \ll_z Y$. For any 
set $X$ and any $k \in \mathbb{N}$, we write $X^k = \{(x_1, \dots, x_k) : x_1, \dots, x_k \in X\}$. Given $k \in \mathbb{N}$, we will use $\bv$ to denote the vector $(v_1, \dots, v_k) \in \mathbb{R}^k$. For a set or statement $E$, we write $\1_E$ for its indicator function.

\section{Preliminaries}

Let $G$ be a finite abelian group, and let $\widehat{G}$ be its group of characters. Given a function $f: G \to \mathbb{C}$, we define its Fourier transform $\widehat{f}: \widehat{G} \to \mathbb{C}$ as
\[ \widehat{f}(\gamma) = \sum_{x \in G} f(x) \gamma(x).   \]
By orthogonality, we get that for all $y \in G$, one has
\[ f(y) = |G|^{-1} \sum_{ \gamma \in \widehat{G}} \widehat{f}(\gamma) \overline{\gamma(y)}. \]
For any function $g : \widehat{G} \to \bC$ and any $q \in [1,\infty)$, we define
\[ \n{g}_q = \big( |G|^{-1} \sum_{\gamma \in \Fr{G}} |g(\gamma)|^q \big)^{1/q} \ \ \text{and} \ \ \n{g}_{\infty} = \max_{\gamma \in \Fr{G}}|g(\gamma)|.  \]
For any functions $f_1, f_2: G \to \mathbb{C}$ and $g_1, g_2: \Fr{G} \to \mathbb{C}$, we define their convolutions $f_1 * f_2: G \to \mathbb{C}$ and $g_1*g_2: \Fr{G} \to \mathbb{C}$ by
\[ (f_1 * f_2)(x) = \sum_{y \in G} f_1(y) f_2(x-y) 
\]
and
\[
(g_1 * g_2)(\gamma) = |G|^{-1}  \sum_{\gamma' \in \Fr{G}} g_1( \gamma') g_2( \gamma \gamma'^{-1}), 
\]
for all $x \in G$ and all $\gamma \in \Fr{G}$.

\subsection{Wrappers}
We now record the notion of wrappers. These were introduced and employed by Semchankau \cite{Sem2025} to analyse various sum--product-type problems in $\mathbb{F}_p$. Let $\tau \in (0,1)$ be a real number, and let
\begin{equation} \label{part1}
    \mathbb{S}^1  = S_1 \cup S_2 \cup \dots \cup S_r 
\end{equation} 
be a partition of $\mathbb{S}^1$ into arcs such that $\mu(S_1) = \dots = \mu(S_{r-1}) = \tau$ and $0 < \mu(S_r) \leq \tau$. Let  $\gamma_1, \dots, \gamma_d$ be pairwise distinct characters in $\widehat{G}$. For any $\bv = (v_1, \dots, v_d) \in [r]^d$, we define the inhomogeneous Bohr set
\begin{equation}
\label{Bohr}
B_{\bv} = B_{\bv, \gamma_1, \dots, \gamma_d, S_1, \dots, S_r} =  \{ x \in G:  \ \gamma_i(x) \in S_{v_i} \ ( 1 \leq i \leq d) \}. 
\end{equation}
For ease of notation, we omit the dependence of $B_\bv$ on the parameters $\gamma_1, \dots, \gamma_d, S_1, \dots, S_r$. Note that
\[ G = \cup_{\bv \in [r]^d} B_\bv.\]
Given any set $X \subseteq [r]^d$, we call the set $V = \cup_{\bv \in X} B_\bv$ a \emph{$(\tau, d)$-wrapper}. 

These are essentially the wrappers introduced by Semchankau \cite{Sem2025}, but our definition is very slightly more general.
We will use wrappers to contain --- i.e. wrap --- certain sets, up to a small number of elements. The advantage of using wrappers is that they have considerable average Fourier decay.

Given real numbers $0 \leq \ell_1 < \ell_2$ and any function $f: G \to [0, \infty)$, we define
\[ \Delta_f(\ell_1, \ell_2) = \{ x \in G : \ell_1 \leq f(x) \leq \ell_2\}. \]
The following is essentially \cite[Corollary 1]{Sem2025}, and can be deduced by combining elementary combinatorial arguments with a result of Croot--\L aba--Sisask \cite[Corollary 3.3]{CLS2013}.

\begin{lemma} 
\label{cor}
Let $f: G \to [0, \infty)$, let $\tau \in (0,1)$, and let $q > 2$ be a real number. Then there exists a positive integer $d \ll q/ \tau^2$, pairwise distinct, non-trivial characters $\gamma_1, \dots, \gamma_d \in \widehat{G}$, a partition of $\mathbb{S}^1$ as in \eqref{part1}, and a set $Z \subseteq G$ with $|Z| \leq e^{-q} |G|$ such that the following holds. For any real numbers $0 \leq \ell_1 < \ell_2$, there exists $X \subseteq [r]^d$ such that
    \[  \Delta_f(\ell_1 + \delta, \ell_2 - \delta) \setminus Z \ \subseteq \ (\cup_{\bv \in X} B_{\bv}) \setminus Z \ \subseteq \ \Delta_f( \ell_1 - \delta, \ell_2 + \delta) \setminus Z, \]
    where $\delta = 20 \n{\Fr{f}}_1 \tau$.
\end{lemma}

We remark that for certain choices of $\ell_1, \ell_2$, one could have $X = \emptyset$.

\subsection{A popular Cauchy--Davenport lemma}

The following is a general Pollard--Kneser-type result due to Green--Ruzsa \cite[Proposition 6.2]{GR2005}. 

\begin{lemma} 
Let $G$ be a finite abelian group, and let $H$ be the largest proper subgroup of $G$. Then, for every non-empty $A, B \subseteq G$ and every $1 \leq t \leq \min \{ |A|, |B|\}$,
\[ 
\sum_{x \in G} \min \{ t, \sum_{\substack{a \in A \\ b \in B}} \1_{a+b = x} \} \geq t \min \{|G|, |A| + |B| - t - |H| \}. 
\]
\end{lemma}

This can be used to establish the following popular Kneser-type inequality, see the derivation of 
\cite[Theorem 4.1]{Lim2025}.

\begin{lemma}
\label{PCD}
For each $s \in \mathbb{N}$ and each $\delta, \kappa >0$, there exist $c,M>0$ such that the following holds. Let $\theta_1, \dots, \theta_s > \del$ with 
\[ 
\theta_1 + \dots + \theta_s \geq 1 + \kappa.
\]
Let $G$ be a finite abelian group $G$ such that its largest proper subgroup $H$ satisfies $|H| \leq |G|/M$. Then, for any $x \in G$ and any $A_1, \dots, A_s \subseteq G$ satisfying $|A_i| \geq \theta_i |G|$ for all $1 \leq i \leq s$,
\[ 
\sum_{a_1 \in A_1, \dots, a_s \in A_s} \1_{a_1 + \dots + a_s = x} \geq c |G|^{s-1}. 
\]
\end{lemma}

\subsection{Deligne's bound}

The following estimate was demonstrated by Deligne. This formulation can be found in the introduction of \cite{Katz}, and is a special case of \cite[Theorem 8.4]{Deligne}.

\begin{lemma} 
\label{DeligneBound}
Let $k \in \bN$, and let $p$ be a prime such that $p \nmid k$. Let $\bF_q$ be a field of characteristic $p$, and let $\psi: \bF_q \to \bS^1$ be a non-trivial additive character. Let $P(x) \in \bF_q[x]$ be a polynomial of degree $k$. Then
\[
\left| \sum_{x \in \bF_q} \psi(P(x)) \right| \le (k-1) \sqrt q.
\]
\end{lemma}

For the rings $\bZ/p^m\bZ$, there are also power-saving exponential sum bounds. The following estimate follows from \cite[Equation 7.9]{Vau1997} and Lemma \ref{DeligneBound}.

\begin{lemma}
\label{DeligneBoundAlt}
Let $k, m \in \bN$, and let $p$ be a prime such that $p > k$. Let $P(x) \in (\bZ/p^m\bZ) [x]$ be a polynomial of degree $k$. Then
\begin{equation*}
\sum_{x\in \bZ/p^m\bZ} e_{p^m}(P(x)) \ll_k p^{m-1/k}.
\end{equation*}
\end{lemma}

\section{Semchankau's wrapping lemma}

Our aim in this section is to record the proof of the following variation of a result due to Semchankau \cite{Sem2025}.

\begin{theorem} \label{sem}
Let $\varepsilon, \delta \in (0,1]$, let $s \geq 3$ be an integer, and let $A_1, \dots, A_s \subseteq G$ satisfy $|A_i| \geq \delta |G|$ for every $1 \leq i \leq s$. Suppose there exists $a \in G$ such that
\[ 
\sum_{a_1 \in A_1, \dots, a_s \in A_s} \mathds{1}_{a_1 + \dots + a_s = a} < \varepsilon |G|^{s-1}. 
\]
Then, for every $1\leq i \leq s$, there exist $\tau_i >0$ and $d_i \in \mathbb{N}$, as well as $W_i, Y_i \subseteq G$ such that $W_i$ is a $(\tau_i, d_i)$-wrapper, 
  \[  \tau_i \gg \eps^{1/(2s)} \delta^{1/2} \eps^{\lceil \eps^{-2} \rceil} ,
  \qquad
  \ d_i \ll \tau_i^{-2} \log (10/\eps)      \]
  and
    \[ |Y_i| \ll \varepsilon^{1/(2s)}|G|,
    \qquad A_i \setminus Y_i \subseteq W_i .  \]
Moreover, there exists $b \in G$ such that
\[ \sum_{w_1 \in W_1, \dots, w_s \in W_s} \mathds{1}_{w_1 + \dots + w_s = b} \ll_s \delta^{-s} \varepsilon^{1/2} |G|^{s-1} .\]
\end{theorem}

We now present the proof of Theorem \ref{sem}. We may assume that $a = 0$ by translating $A_1$ appropriately. Let $n \geq 2$ be an integer to be chosen later, and let 
   \[ \Gamma = \{ \gamma \in \Fr{G} : |\Fr{\1_{A_1}}(\gamma)| > |G|/n^{1/2} \}. \]
   Since $s \geq 3$, 
   \begin{align} \label{wgon}
  & |G|^{-1}  \sum_{\gamma \in \Fr{G} \setminus \Gamma}     |\Fr{\1_{A_1}}(\gamma) \cdots \Fr{\1_{A_s}}(\gamma) | \leq |G|^{s-4}\sum_{\gamma \in \Fr{G} \setminus \Gamma} |\Fr{\1_{A_1}}(\gamma) \Fr{\1_{A_2}}(\gamma) \Fr{\1_{A_3}}(\gamma) |  \nonumber \\
   & \leq |G|^{s-4} (\sum_{\gamma \in \Fr{G}} |\Fr{\1_{A_2}}(\gamma)|^2 )^{1/2} (\sum_{\gamma \in \Fr{G}} |\Fr{\1_{A_3}}(\gamma)|^2 )^{1/2} \max_{\gamma \in \Fr{G} \setminus \Gamma} |\Fr{\1_{A_1}}(\gamma)| \nonumber \\
   & \leq |G|^{s-1}/n^{1/2}. 
   \end{align}
   By orthogonality, one has
   \[ \# \Gamma |G|^2/n \leq \sum_{\gamma \in \Fr{G}} |\Fr{\1_{A_1}}(\gamma)|^2 = |G| \# A_1 \leq |G|^2,\]
   whence $\# \Gamma \leq n$.

Let $\sigma \in (0,1)$ be a parameter to be chosen later, let $S_{\sigma}$ be an arc of length $2\sigma$ centred at $1 \in \bS^1$, and let $B$ be the Bohr set
   \[ B = \{ x \in G   :  \gamma(x) \in S_{\sigma} \ \text{for all} \ \gamma \in \Gamma \}. \]
The Bohr set $B$ has rank $|\Gamma| \leq n$, and so
   \begin{equation} \label{bohrsize}
       |B| \geq \sigma^{|\Gamma|} |G| \geq \sigma^n |G|,
   \end{equation}
   see \cite[Lemma 4.20]{TV2006}. Moreover, for any $\gamma \in \Gamma$, 
   \begin{equation} \label{bohrfoureq}
   \Fr{\1_B}(\gamma) = \sum_{b \in B} \gamma(b) =  |B| ( 1 + O(\sigma  )).  
   \end{equation}

Now for every $1 \leq i \leq s$, define 
\[ f_i(x) = |B|^{-1} (\1_{A_i} * \1_B )(x)  = |B|^{-1} \sum_{y \in G} \1_{A_i}(y) \1_B(x-y) \]
for all $x \in G$. These functions approximate the additive structure of $\1_{A_i}$ on the physical side while having smoother Fourier $L^1$ norm, see \eqref{pst} and \eqref{l1four} respectively.
We first prove the former statement, and so, note that
\[ \Fr{f_i}(\gamma) = |B|^{-1} \Fr{\1_{A_i}}(\gamma) \Fr{\1_B}(\gamma) .\]
This, in turn, implies that $|\Fr{f_i}(\gamma)| \leq |G|/n^{1/2}$ for all $\gamma \notin \Gamma,$ while $\Fr{f_i}(\gamma) = \Fr{\1_{A_i}}(\gamma)( 1+ O(\sigma))$  for all $\gamma \in \Gamma$, with the latter equality following from \eqref{bohrfoureq}.  Furthermore,  \eqref{wgon} implies that
\begin{align*}
|G|^{-1} \sum_{\gamma \notin \Gamma} | \Fr{f_1}(\gamma) \cdots \Fr{f_s}(\gamma)  |  &\leq |G|^{-1} \sum_{\gamma \notin \Gamma}  |\Fr{\1_{A_1}}(\gamma) \dots \Fr{\1_{A_s}}(\gamma) | \\&
\leq |G|^{s-1}/n^{1/2}. 
\end{align*}
Consequently,
\begin{align}  
\label{pst}
&\sum_{x_1, \dots, x_s} f_1(x_1) \cdots f_s(x_s) \1_{x_1 + \dots + x_s = 0} 
= |G|^{-1} \sum_{\gamma}  \Fr{f_1}(\gamma) \cdots \Fr{f_s}(\gamma) \nonumber \\
& = |G|^{-1} \sum_{\gamma \in \Gamma} \Fr{\1_{A_1}}(\gamma) \cdots \Fr{\1_{A_s}}(\gamma) ( 1+ O(\sigma))^s + O(|G|^{s-1}/n^{1/2}) \nonumber \\
& = |G|^{-1} \sum_{\gamma \in \Gamma} \Fr{\1_{A_1}}(\gamma) \cdots \Fr{\1_{A_s}}(\gamma) + O_s(|G|^{s-1} ( n^{-1/2} + \sigma)) \nonumber \\
&  = \sum_{x_1, \dots, x_s} \1_{A_1}(x_1) \cdots \1_{A_s}(x_s) \1_{x_1 + \dots + x_s = 0} + O_s(|G|^{s-1} ( n^{-1/2} + \sigma)) \nonumber \\
& < \eps |G|^{s-1} + O_s( |G|^{s-1} n^{-1/2} + \sigma)),
\end{align}
where the penultimate step follows from \eqref{wgon} and orthogonality. As for the Fourier properties of $f_i$, note by a straightforward application of the Cauchy--Schwarz inequality and orthogonality that
\begin{equation} \label{l1four}
\n{ \Fr{f_i}}_1 = |G|^{-1}|B|^{-1} \sum_{\gamma \in \Fr{G}} | \Fr{\1_{A_i}}(\gamma)  \Fr{\1_{B}}(\gamma) |  \leq |A_i|^{1/2} |B|^{-1/2}.
\end{equation}

Now let $q>2$ and $\alpha \in (0,1)$ be parameters and, for every $1 \leq i \leq s$, let
\begin{equation}
\label{choice}
\eta_i = \frac{\alpha |A_i|}{20|G|} \quad \text{and} \quad \tau_i = \frac{\alpha |A_i|^{1/2}|B|^{1/2}}{800|G|}  
\end{equation}
be reals lying in $(0,1)$. For each $1 \leq i \leq s$, we apply Lemma~\ref{cor} to obtain a positive integer $d_i \ll q / \tau_i^2$ and a set $Z_i \subseteq G$ with $|Z_i| \leq e^{-q} |G|$ such that, upon setting $\ell_1 = \eta_i$ and $\ell_2 = 1 + 40 \n{\Fr{f_i}}_1 \tau_i$, we obtain a $(\tau_i, d_i)$-wrapper $V_i$ such that
\begin{align} \label{magma2}
\Delta_{f_i}(\eta_i + 20 \n{\Fr{f_i}}_1 \tau_i, & \  1 + 20 \n{\Fr{f_i}}_1 \tau_i) \setminus Z_i \ \subseteq \ V_i \setminus Z_i \nonumber \\
& \subseteq \ \Delta_{f_i}( \eta_i - 20 \n{\Fr{f_i}}_1 \tau_i, \ 1 + 60 \n{\Fr{f_i}}_1 \tau_i) \setminus Z_i.
\end{align}
Let us also record some useful properties that we will need later. By \eqref{l1four},
\begin{equation} \label{magma}
20 \n{\Fr{f_i}}_1 \tau_i \leq \eta_i/2. 
\end{equation}
This, along with our choice of $\eta_i, \tau_i$, implies that
\begin{align} \label{ltruse}
\sum_{x \in G: f_i(x) < \eta_i + 20 \n{\Fr{f_i}}_1 \tau_i } f_i(x) 
& < (\eta_i + 20 \n{\Fr{f_i}}_1 \tau_i  )|G| < 2 \eta_i |G|  \nonumber \\
& = \alpha |A_i|/10 < |A_i| = \sum_{x \in G} f_i(x),
\end{align}
whence $\Delta_{f_i}( \eta_i + 20 \n{\Fr{f_i}}_1 \tau_i,   1) \neq \emptyset$. Finally, the second inclusion in \eqref{magma2} combines with \eqref{magma} to furnish
\begin{equation} \label{grd}
\1_{V_i} \leq  \frac{f_i}{\eta_i - 20 \n{\Fr{f}}_1 \tau_i} + \1_{Z_i} \leq \frac{2f_i}{\eta_i} + \1_{Z_i} \leq 2\eta_i^{-1} ( f_i + \1_{Z_i}).
\end{equation}

We now establish some desirable additive properties of $V_1, \dots, V_s$. Inequality \eqref{grd} gives us
\begin{align*}
   &  \sum_{v_1 \in V_1 ,\dots, v_s \in V_s}  \1_{v_1 + \dots + v_s = 0} 
    \\ &\leq \sum_{x_1, \dots, x_s \in G} \prod_{i=1}^s  2 \eta_i^{-1} (f_i(x_i) + \1_{Z_i}(x_i) ) \1_{x_1 + \dots + x_s = 0}  \\
   &\ll_s (\eta_1 \cdots \eta_s)^{-1} \left( \sum_{x_1, \dots, x_s \in G} f_1(x_1) \cdots f_s(x_s)  \1_{x_1 + \dots + x_s = 0} + e^{-q} |G|^{s-1}  \right) .
\end{align*}
The assumption that $|A_i| \ge \del |G|$ along with the definition of $\eta_i$ in \eqref{choice} imply that $\eta_i \gg \alpha \delta$ for every $1 \leq i \leq s$. This in turn combines with the preceding inequality and \eqref{pst} to deliver the estimate
\[    \sum_{v_1 \in V_1 ,\dots, v_s \in V_s}  \1_{v_1 + \dots + v_s = 0}  \ll_s \alpha^{-s} \delta^{-s} |G|^{s-1} ( \eps + n^{-1/2} + \sigma + e^{-q} ).  \]
Setting 
\begin{equation}
\label{params}
n = \lceil \eps^{-2} \rceil, \quad
\sigma = \eps \quad \text{and} \quad q = \log (10/\eps)
\end{equation}
yields
\begin{equation} \label{magma3}
 \sum_{v_1 \in V_1 ,\dots, v_s \in V_s}  \1_{v_1 + \dots + v_s = 0} \ll_s \eps \alpha^{-s} \delta^{-s} |G|^{s-1}. 
\end{equation}

Given $1 \leq i \leq s$ and $b \in B$, define
\[ A_{i,b} = \{ a \in A_i : f_i(a+b) < \eta_i + 20 \n{\Fr{f_i}}_1 \tau_i \} .\] 
Let $b_i \in B$ satisfy $|A_{i,b_i}| \leq |A_{i,b}|$ for all $b \in B$. A standard averaging argument, together with \eqref{ltruse}, gives
\begin{align*}
    |A_{i,b_i}| \leq |B|^{-1} \sum_{b \in B} |A_{i,b}| = \sum_{x \in G: f_i(x) < \eta_i + 20 \n{\Fr{f_i}}_1 \tau_i } f_i(x) < \alpha |A_i|/10. 
\end{align*}
Now, if $a \in A_i \setminus A_{i,b_i}$, then $f_i(a+ b_i) \geq \eta_i + 20 \n{\Fr{f_i}}_1 \tau_i$, which in turn combines with the first inclusion in \eqref{magma2} to tell us that either $a+ b_i \in Z_i$ or $a+b_i \in V_i$. For every $1 \leq i \leq s$, we set
\[ W_i = V_i - b_i \quad \text{and} \quad Y_i = A_{i, b_i} \cup (Z_i - b_i) . \]
Then
\[ |Y_i| \leq \alpha |A_i|/10 + e^{-q}|G| \ll \alpha |G| + \eps |G| .\]
Setting $\alpha = \eps^{1/(2s)}$ delivers the desired upper bound for $|Y_i|$.  Moreover, note that $A_i \setminus Y_i \subseteq W_i$ and
\begin{align*}
\sum_{w_1, \in W_1, \dots, w_s \in W_s} \1_{w_1 + \dots + w_s = -b_1 - \dots - b_s} &  = \sum_{v_1 \in V_1, \dots, v_s \in V_s} \1_{v_1 + \dots + v_s = 0} \\
& \ll \eps^{1/2} \delta^{-s} |G|^{s-1}, 
\end{align*}
using \eqref{magma3} and the fact that $\alpha = \eps^{1/(2s)}$.

The simple observation that translates of $(\tau_i,d_i)$-wrappers are also $(\tau_i,d_i)$-wrappers ensures that $W_i$ is a $(\tau_i,d_i)$-wrapper. Our choice of $\tau_i$ in \eqref{choice} along with \eqref{bohrsize} and \eqref{params}  implies that
\[ \tau_i  \gg \alpha \delta^{1/2} \sigma^{n/2} \gg \eps^{1/(2s)} \delta^{1/2} \eps^{\lceil \eps^{-2} \rceil} .  \]
Finally, since $d_i \ll q/\tau_i^2$, our choice of $q$ dispenses the estimate
\[  d_i \ll \tau_i^{-2} \log(10/\eps) .\]
This concludes the proof of Theorem \ref{sem}.

\section{Finite fields}

This section will culminate in a proof of Theorem \ref{thm1}. Here $q = p^m$, for some prime $p$ and some $m \in \bN$.
For $\alp,x \in \bF_q$, let
\[
\psi_\alp(x) = e_p(\Tr(\alp x))
= \psi_1(\alp x).
\]

\subsection{Equidistribution of powers}

\begin{lemma}   \label{writ}
Let $r$ be a positive integer, let $X_1, \dots, X_r$ be 
non-empty subsets of $\bF_q$, and let $1 \leq i_1 < i_2 < \dots < i_r < p$ be integers. Then
\begin{align*}
&\#
\{ n \in \bF_q:  n^{i_j} \in X_j \  ( 1 \leq j \leq r )\} - \ \frac{|X_1| \cdots |X_r|}{q^{r-1}}  \\ &\ll_{i_r}  \n{\Fr{\mathds{1}_{X_1}}}_1 \cdots \n{\Fr{ \mathds{1}_{X_r}}}_1 \sqrt q. 
\end{align*}
\end{lemma}

\begin{proof}
Put $X = X_1 \times \cdots \times X_r$. 
By orthogonality,
\begin{align*}
& \#
\{ n \in \bF_q   :  n^{i_j} \in X_j \  ( 1 \leq j \leq r )\}  \\
&= \sum_{n \in \bF_q} \sum_{\bx \in X} \mathds{1}_{n^{i_1} =x_1} \cdots \mathds{1}_{n^{i_r} = x_r} \\
&= q^{-r} \sum_{\balp \in \bF_q^r} \sum_{\bx \in X} 
\psi_1(-\balp \cdot \bx) \sum_{n \in \bF_q}  \psi_1(\alp_1 n^{i_1}+ \cdots + \alp_r n^{i_r}) \\
&= q^{-r} \sum_{\balp \in \bF_q^r}
\Fr{\mathds{1}_{X_1}}(\psi_{-\alpha_1}) \cdots \Fr{\mathds{1}_{X_r}}(\psi_{-\alpha_r})  \sum_{n \in \bF_q}  \psi_1(\alpha_1 n^{i_1} + \cdots + \alpha_r n^{i_r}).
\end{align*}
Thus, by Lemma \ref{DeligneBound},
\begin{align*}
&\bigg| \# \{ n \in \bF_q:  n^{i_j} \in X_j \  ( 1 \leq j \leq r )\}  - \frac{|X_1| \cdots |X_r|}{q^{r-1}} \bigg|\\
& \leq q^{-r} \sum_{\bzero \ne \balp \in \bF_q^r}  |\Fr{\mathds{1}_{X_1}}(-\alpha_1)\cdots \Fr{\mathds{1}_{X_1}}(-\alpha_r)\sum_{n \in \bF_q} \psi_1(\alpha_1 n^{i_1} + \cdots + \alpha_r n^{i_r}) | \\
& \leq  \|\Fr{\mathds{1}_{X_1}}\|_1 \cdots \|{\Fr{\mathds{1}_{X_r}}}\|_1 i_r \sqrt q.\qedhere
\end{align*}
\end{proof}

\subsection{Average Fourier decay}

As an additive group, we can identify $\bF_q$ with $G = \bF_p^m$. In this subsection, we use $\langle \cdot , \cdot \rangle$ interchangeably with the dot product on $G$. Let $\br \in G$ and
\[
\psi_\br(\bx) = e_p(\br \cdot \bx).
\]
Bohr sets of rank 1 for this additive character have the form
\[
B = \left \{ \bx \in G: 
\left \| \frac{\br \cdot \bx - c}p \right \| \le \del \right \},
\]
where $c \in \bZ$ and $0 < \del \le 1/2$.

In the case $m=1$, it follows from \cite[Proposition 3.2]{Sem2025}
and the proof of \cite[Proposition 3.3]{Sem2025}
that 
\begin{equation}
\label{AverageDecay}
\| \Fr{\1_B} \|_1 \ll \log p.
\end{equation}
We now extend this to the case of general $m$.

\begin{lem}\label{Fourier decay finite fields} We have
\eqref{AverageDecay}.
\end{lem}

\begin{proof}
First suppose $\br = \bzero$. Then $B$ is either empty or $G$. If $B$ is empty, then $\| \Fr{\1_B} \|_1 = 0.$ If $B = G$, then
\[
\| \Fr{\1_B} \|_1 = |G|^{-1} |\Fr{\1_G}(\bzero)| = 1.
\]
It suffices to consider two other cases for $\br$, namely:
\begin{enumerate} [I.]
\item $\br \cdot \br \ne 0$
\item $\br \cdot \br = 0$.
\end{enumerate}
Let $\bl \in G \setminus H$, where
\[
H = \{ \bx \in G: \br \cdot \bx = 0 \}.
\]
In Case I, we choose $\bl = \br$ for simplicity.

We decompose
\[
G = H \oplus \bF_p \bl.
\]
This restricts to a decomposition
\[
B = H + U \bl,
\]
where
\[
U = \left \{ u \in \bF_p: 
\left \| \frac{u \bl \cdot \br - c}p \right \| \le \del \right \},
\]
and each element of $B$ is  $\bh + u \bl$ for some unique pair $(\bh, u) \in H \times U$.
Now
\[
\Fr{\1_B}(\bxi) = \sum_{\bx \in B} e_p(\bx \cdot \bxi),
\]
so
\begin{align*}
\| \Fr{\1_B} \|_1 &= p^{-m} 
\sum_{\substack{\bxi_0 \in H \\ v \in \bF_p}}
\left|
\sum_{u \in U} \sum_{\bh \in H}
e_p( \langle \bh + u \bl, \bxi_0 + v \bl \rangle) 
\right| \\
&= p^{-m} 
\sum_{\substack{\bxi_0 \in H \\ v \in \bF_p}}
\left|
\sum_{u \in U} e_p(u \langle \bl, \bxi_0 + v \bl \rangle)
\sum_{\bh \in H}
e_p( \langle \bh, \bxi_0 + v \bl \rangle)
\right|.
\end{align*}
The inner sum vanishes unless $\bxi_0 + v \bl = \lam \br$ for some $\lam \in \bF_p$. 

\subsection*{Case I}

In this case $\bl = \br$.
For the inner sum not to vanish, we must have $\bxi_0 \| \br$ and hence $\bxi_0 = \bzero$, since $\br \cdot \bxi_0 = 0$ and $\br \cdot \br \ne 0$. Consequently,
\[
\| \Fr{\1_B} \|_1 
= p^{-1} \sum_{v \in \bF_p} \left| \sum_{u \in U} e_p(uv \br \cdot \br) \right| = p^{-1} \sum_{v \in \bF_p} \left| \sum_{u \in U} e_p(uv) \right|.
\]
This is $O(\log p)$, as we saw in the case $m=1$.

\subsection*{Case II}

As $\bl \notin H$ and $\bxi_0, \br \in H$, we cannot have $\bxi_0 + v \bl = \lam \br$ unless $v = 0$ and $\bxi_0 = \lam \br$. Therefore
\[
\| \Fr{\1_B} \|_1 = p^{-1} \sum_{\lam \in \bF_p} \left| \sum_{u \in U} e_p(u \lam \bl \cdot \br) \right| = p^{-1} \sum_{\lam \in \bF_p} \left| \sum_{u \in U} e_p(u \lam) \right|.
\]
This is $O(\log p)$, as in Case I.
\end{proof}

We can bootstrap this to infer average Fourier decay for wrappers.

\begin{lemma} \label{wrapperfourier}
Let $W$ be a $(\tau,d)$-wrapper in $\bF_q$. Then
\[
\| \Fr{\1_W} \|_1 \ll_{\tau,d} (\log p)^d.
\]
\end{lemma}

\begin{proof}
By definition, the wrapper $W$ is a disjoint union of $O_{\tau,d}(1)$ many sets
$B = B_\bv$ as given by \eqref{Bohr}, and so, by a simple application of triangle inequality, it suffices to prove the desired result for the set $B_{\bv}$. In this case, note that we can write
\[
B = B_1 \cap \cdots \cap B_d,
\]
for some Bohr sets $B_1, \ldots, B_d$ of rank 1. Finally, Young's convolution inequality and Lemma \ref{AverageDecay} deliver
\begin{align*}
\| \Fr{\1_B} \|_1
&= \| \Fr{\1_{B_1} \cdots \1_{B_r}} \|_1
= \| \Fr{\1_{B_1}} * \cdots *\Fr{\1_{B_r}} \|_1 \\ &\le \| \Fr{\1_{B_1}} \|_1 \cdots \| \Fr{\1_{B_r}} \|_1 \ll (\log p)^d. \qedhere
\end{align*}
\end{proof}

\subsection{Dense Fermat over finite fields} 
We now establish Theorem~\ref{thm1}. Assume for a contradiction that there are fewer than $\eps q^{s-1}$ many solutions $\bx \in A^s$ to 
\eqref{MainEq}.
Define 
\[
A_n = \{ c_n a^{i_j} : a \in A\}
\]
for all $n \in I_j$ and all $1 \leq j \leq r$. Note that $|A_n| \geq |A|/i_r \gg_{i_r, s} q$. We may now apply Theorem \ref{sem} to find, for every $1 \leq i \leq s$, a $(\tau_i, d_i)$-wrapper $W_i$ with $\tau_i \gg_{\eps, s, i_r} 1$ and $d_i \ll_{\eps, s, i_r} 1$ such that 
\[
|A_i \setminus W_i| \ll \eps^{1/(2s)}q
\]
and 
\begin{equation} \label{flyingmicrotonalbanana1}
\sum_{w_1 \in W_1, \dots, w_s \in W_s} \1_{w_1 + \dots + w_s = b} \ll_{s, i_r} \eps^{1/2} q^{s-1},
\end{equation}
for some $b \in \mathbb{F}_q$. 

We now claim that 
\begin{equation} \label{flyingmicrotonalbanana2}
|W_1| + \dots + |W_s| < \left(1 + \frac{\kappa}{10}\right)q. 
\end{equation}
We prove this claim via contradiction, so let us assume that it does not hold. Let $C$ be a large, positive constant. As $\eps$ is sufficiently small, we have 
\[ 
|W_i| \geq |A_i| - C\eps^{1/(2s)} q \gg_{i_r,s} q  
\]
for all $1 \leq i \leq s$. We may now apply Lemma \ref{PCD}, giving
\[ 
\sum_{w_1 \in W_1, \dots, w_s \in W_s} \1_{w_1 + \dots + w_s} \gg_{s, i_r, \kappa} q^{s-1}. 
\]
As $\eps$ is sufficiently small in terms of $s, i_r, \kappa$, this contradicts \eqref{flyingmicrotonalbanana1}. 

We may therefore assume that \eqref{flyingmicrotonalbanana2} holds. For every $\bn \in I_1 \times \dots \times I_r$, let 
\[ 
X_{\bn} =\{ x \in \mathbb{F}_q : c_{n_j} x^{i_j} \in W_{n_j} \ (1 \leq j \leq r) \} . 
\]
By Lemma \ref{wrapperfourier} and a change of variables, we note that if $W$ is a $(\tau,d)$-wrapper in $\bF_q$ and $c \in \bF_q^\times$ then
\[
\| \Fr{\1_{cW}} \|_1 \ll_{\tau,d} (\log p)^d.
\]
Combining this with Lemma \ref{writ} yields
\begin{align*}
|X_{\bn}|  - \frac{|W_{n_1}| \cdots |W_{n_r}|}{q^{r-1}} 
 & \ll_{s, i_r, \eps} ( \log p)^{O_{\eps, s,i_r}(1)} q^{1/2}. 
 \end{align*}
Since $|A_i \setminus W_i| \ll \eps^{1/(2s)} q$ for all $1 \leq i \leq s$,
\[ |A| \leq |X_{\bn}|  + O_{i_r, s}(\eps^{1/(2s)} q) .\]
Coupling this with the preceding bound, and then averaging over all $\bn \in I_1 \times \dots \times I_r$, gives
\begin{align} \label{flyingmicrotonalbanana3}
\frac{|A|}{q} & \leq \frac{1}{k_1 \cdots k_r} \sum_{\bn \in I_1 \times \dots\times I_r} \frac{|W_{n_1}| \cdots |W_{n_r}| }{q^{r}} \nonumber \\
& \qquad + O_{s, i_r}(\eps^{1/(2s)}) + O_{\eps, i_r, s}((\log p)^{O_{\eps,s,i_r}(1)} q^{-1/2}). 
\end{align}

Clearly, we may assume that $\kap$ is sufficiently small in terms of $r$. We may also assume that $\eps$ is a function of $s, i_r, \kap$, as the result would then follow for all smaller values of $\eps$. This enables us to subsume any dependence on $\eps$ into the dependence on $s, i_r, \kap$.

By the AM--GM inequality and \eqref{flyingmicrotonalbanana2}, the first term on the right-hand side of \eqref{flyingmicrotonalbanana3}
is bounded above by
\begin{align*}
\frac{1}{ q^r  k_1 \cdots k_r} \left(  r^{-1} \sum_{j=1}^r \sum_{n_j \in I_j} |W_{n_j}| \right)^r & \leq \frac{(1 + \kappa/10)^r}{k_1 \cdots k_r r^r} \\
& \leq \frac{1 +r\kappa/5}{k_1 \cdots k_r r^r} \le  \prod_{1 \le j \le r} \frac1{rk_j} +  \frac{\kappa}{5}.
\end{align*}
The other terms on the right-hand side of \eqref{flyingmicrotonalbanana3} are smaller than $\kappa/5$. Therefore 
\[
|A| < \left( \prod_{1 \le j \le r} \frac1{rk_j} +  \frac{3\kappa}{5} \right)q,
\]
contradicting the hypothesis \eqref{DensityAssumption}. Hence, we must have at least $\eps q^{s-1}$ many solutions to \eqref{MainEq} with $x_1, \dots, x_s \in A$.

\section{Finite, cyclic rings}

In this section, we will first prove that the wrappers are nicely additively structured, in the sense that the powers are equidistributed among them. Then we will use this property to deduce Theorem \ref{thm2}. 

\subsection{Equidistribution of powers in wrappers}

Let $N$ be a positive integer, let $\beta_1,\beta_2,\dots,\beta_d\in \bZ/N\bZ$ be frequencies, and let $\bI_1,\bI_2, \dots,\bI_d$ be real intervals. The inhomogeneous Bohr set $$B=B(\beta_1,\beta_2,\dots,\beta_d; \bI_1,\bI_2,\dots,\bI_d)$$ is defined to be
\begin{equation*}
=\left\{x\in \bZ/N\bZ: \frac{\beta_j x}{N}\in \bI_j \text{\;$\mathrm{mod}\;1$ for all $1\leq j\leq d$}\right\}.
\end{equation*}
We refer to $d$ as the \emph{rank} of $B$. (Strictly speaking, the rank is a property of the Bohr set data, rather than a property of the set.)

The following lemma, which is in the same spirit as Lemma \ref{Fourier decay finite fields}, asserts that the inhomogeneous Bohr sets in $\bZ/p^m\bZ$ have considerable average Fourier decay.

\begin{lemma}
\label{Fourier decay prime power}
Let $p$ be a prime, let $m\in \bN$, and let $B$ be an inhomogeneous Bohr set of rank $d$ in $\bZ/p^m\bZ$. Then
    \begin{equation*}
        \n{\Fr{\1_B}}_1 \leq (C\log p^m)^d,
    \end{equation*}
    where $C$ is some absolute constant.
\end{lemma}

\begin{proof}
First, we consider the case where the Bohr set $B$ has rank $1$. Say $B=B(p^\ell \beta;\bI)$, where $\beta$ is not divisible by $p$ and $\ell < m$.

Note that for $x\in \bZ/p^m\bZ$, whether $p^\ell \beta x/p^m=\beta x/p^{m-\ell}$ is in $\bI$ depends only on the residue class of $x \mmod p^{m-\ell}$. Thus, the Bohr set $B$ is of the form $Q+H$, where $Q$ is an arithmetic progression $\mmod p^{m-\ell}$ with common difference $z$ not divisible by $p$, and $H$ is the subgroup of $\bZ/p^m\bZ$ generated by the element $p^{m-\ell}$. With this description, we can compute $\Fr{\1_B}$ explicitly. Indeed,
\begin{align*}
\Fr{\1_B}(\xi) &=
\sum_{q\in Q} \ \sum_{0\leq x< p^\ell} e_{p^m}((q+xp^{m-\ell})\xi) = \sum_{q\in Q} e_{p^m}(q\xi) \sum_{0\leq x< p^\ell} e_{p^\ell}(x\xi) \\
&=\begin{cases}
0, & \text{if }p^l \nmid \xi\\
p^\ell \sum_{q\in Q}e_{p^m}(q\xi) , & \text{if }p^\ell \mid \xi
\end{cases}\\
&\ll\begin{cases}
0, & \text{if }p^\ell \nmid \xi\\
p^\ell \n{\frac{\xi z}{p^m}}_{\bR/\bZ}^{-1}, & \text{if }p^\ell \mid \xi.
\end{cases}
\end{align*}
Thus, we have
\begin{align*}
\sum_{\xi \in \bZ/p^m\bZ}|\Fr{\1_B}(\xi)|&\ll p^m + p^\ell \sum_{1\leq \eta< p^{m-l}} \left\| \frac{\eta z}{p^{m-\ell}} \right\|_{\bR/\bZ}^{-1} \\
&\ll p^m + p^\ell \sum_{1\leq \eta < p^{m-\ell}} \frac{p^{m-l}}{\eta} \ll p^m \log(p^m),
\end{align*}
which gives the desired bound.

For higher-rank Bohr sets, we observe that $\n{\Fr{\1_B1_{B'}}}_1\leq\n{\Fr{\1_B}}_1\n{\Fr{\1_{B'}}}_1$ for any Bohr sets $B,B'$ by Young's inequality. Hence, combining this inequality and the fact that higher-rank Bohr sets are intersections of rank $1$ Bohr sets completes the proof.
\end{proof}

Next, we will prove that the powers are equidistributed among the inhomogeneous Bohr sets in $\bZ/p^m\bZ$.

\begin{prop}
\label{equidistribution prime power}
Let $i_1< \dots<i_r$ be positive integers. Let $p>i_r$ be prime, let $m\in \bN$, and let $B_1, \ldots, B_r$ be inhomogeneous Bohr sets of rank $d$ in $\bZ/p^m\bZ$. Let $c_1, \ldots,c_r\in (\bZ/p^m\bZ)^{\times}$. Then
\begin{align*}
\sum_{x\in \bZ/p^m\bZ}
\1_{B_1}(c_1 x^{i_1}) \cdots \1_{B_r}(c_r x^{i_r}) &= |B_1| \cdots |B_r|p^{-m(r-1)} \\
&\quad+ O_{i_r,d}(p^{m-1/i_r} \log^{dr} p^m).
\end{align*}
\end{prop}

\begin{proof}
By Fourier expansion and Lemma \ref{DeligneBoundAlt}, 
\begin{align*}
&\sum_{x\in \bZ/p^m\bZ} \1_{B_1}(c_1 x^{i_1}) \cdots \1_{B_r}(c_r x^{i_r}) \\
&=p^{-mr}\sum_{x\in \bZ/p^m\bZ} \ \sum_{\xi_1,\ldots,\xi_r\in \bZ/p^m\bZ}\Fr{\1_{B_1}}(\xi_1) \cdots \Fr{\1_{B_r}}(\xi_r) \\
&\qquad \qquad \qquad \qquad \qquad \qquad \cdot e_{p^m}(-\xi_1 c_1 x^{i_1}- \ \cdots \ -\xi_r c_r x^{i_r})\\
&=|B_1| \cdots |B_r|p^{-m(r-1)}+O_{i_r}(p^{m-1/i_r}\n{\Fr{\1_{B_1}}}_1 \cdots \n{\Fr{\1_{B_r}}}_1).
\end{align*}
Thus, applying Lemma \ref{Fourier decay prime power} gives the desired error term.
\end{proof}

To extend Proposition \ref{equidistribution prime power} to general moduli $N=p_1^{m_1} \cdots p_t^{m_t}$, such Fourier-analytic arguments are --- perhaps surprisingly --- less effective. The reason is that that the analysis of Hardy--Littlewood major arcs in $\bZ/N\bZ$ is more delicate; on some non-zero major arcs we only save $\sqrt{p_1}$ over the trivial exponential sum bound, but $p_1$ might be considerably smaller than other prime factors, rendering this saving not as large as we would like. To overcome this issue, we argue in the physical space and utilise the product structure of $\bZ/N\bZ$ to deduce the following proposition. 

\begin{prop}
\label{equidistribution mod N}
Let $i_1< \dots<i_r$ be positive integers. Let $$N=p_1^{m_1} \cdots p_t^{m_t} \quad \text{with} \quad p_1,\ldots,p_t>i_r,$$
and let $B_1,B_2,\ldots,B_r$ be inhomogeneous Bohr sets of rank $d$ in $\bZ/N\bZ$. Let $c_1,\ldots,c_r\in (\bZ/N\bZ)^{\times}$. Then
\begin{align*}
\sum_{x\in \bZ/N\bZ} \1_{B_1}(c_1 x^{i_1}) \cdots \1_{B_r}(c_r x^{i_r}) &= |B_1| \cdots |B_r| N^{-(r-1)} \\ &\quad +O_{i_r,d}\left(N\sum_{1\leq j\leq t}\frac{\log^{dr}p_j^{m_j}}{p_j^{1/i_r}}\right).
\end{align*}
\end{prop}

\begin{proof}
We will prove Proposition \ref{equidistribution mod N} by induction on the number of prime factors $t$. The base case $t=1$ is given by Proposition \ref{equidistribution prime power}.

When $t>1$, write $N=p_1^{m_1} N'$. Let $f_j= \1_{B_j}$ for $1\leq j\leq r$. We can regard $f_j$ as functions on $\bZ/p_1^{m_1}\bZ\times\bZ/N'\bZ$, by the Chinese remainder theorem. Note that
\begin{equation}
\label{eq: equidistribution 1}
\begin{split}
&\sum_{x\in \bZ/N\bZ}\1_{B_1}(c_1 x^{i_1})\cdots\1_{B_r}(c_r x^{i_r})\\
&=\sum_{z\in \bZ/N'\bZ} \ \sum_{y\in \bZ/p_1^{m_1}\bZ}f_1(c_1 y^{i_1},c_1 z^{i_1}) \cdots f_r(c_r y^{i_r},c_r z^{i_r}).
\end{split}   
\end{equation}

The key observation is that a slice of a Bohr set is still a Bohr set of the same rank, that is, we can regard $f_j(\cdot,v)$ as the characteristic function of a Bohr set in $\mathbb{Z}/p_1^{m_1}\bZ$, for $v\in \bZ/N'\bZ$ and $1\leq j\leq r$. To justify this observation, note that the canonical isomorphism $$\mathbb{Z}/p_1^{m_1}\bZ\times\mathbb{Z}/N'\bZ\simeq \mathbb{Z}/N\bZ$$ is given by $(u,v)\longmapsto uN'\Gamma+vp_1^{m_1}\gamma$, where $\Gamma$ is the multiplicative inverse of $N'$ in $\mathbb{Z}/p_1^{m_1}\bZ$ and $\gamma$ is the multiplicative inverse of $p_1^{m_1}$ in $\mathbb{Z}/N'\bZ$. Now, for any $\beta\in \bZ/N\bZ$ and interval $\bI$, we have
\begin{equation*}
\frac{\beta(uN'\Gamma+vp_1^{m_1}\gamma)}{N}\in \bI \iff \frac{\beta\Gamma u}{p_1^{m_1}}\in \bI-\frac{\beta v \gamma}{N'},
\end{equation*}
hence $f_j(\cdot,v)$ is the characteristic function of a Bohr set of the same rank as $B_j$ in $\mathbb{Z}/p_1^{m_1}\bZ$.

By Proposition \ref{equidistribution prime power}, the expression \eqref{eq: equidistribution 1} is equal to
\begin{align}
\notag
\sum_{z\in \bZ/N'\bZ} \Bigg(
p_1^{-m_1(r-1)} \prod_{1\leq j\leq r} \ \sum_{\fy_j\in \bZ/p_1^{m_1}\bZ}
&f_j(\fy_j,c_j z^{i_j}) \\
&+O_{i_r,d}(p_1^{m_1-1/i_r}\log^{dr}p_1^{m_1})
\Bigg).
\label{eq: equidistribution 2}
\end{align}
Summing over $\bZ/N'\bZ$, the error term in \eqref{eq: equidistribution 2} is $O_{i_r,d}(Np_1^{-1/i_r}\log^{dr}p_1^{m_1})$. For the main term, we apply the induction hypothesis for each $\fy_1, \ldots,\fy_r$ --- once again, we use the fact that the slicing of a Bohr set is still a Bohr set of the same rank --- to obtain
\begin{equation}
\label{eq: equidistribution 3}
\begin{split}
&p_1^{-m_1(r-1)}\sum_{\fy_1,...,\fy_r\in \bZ/p_1^{m_1}\bZ} \ \sum_{z\in \bZ/N'\bZ}f_1(\fy_1,c_1 z^{i_1}) \cdots f_r(\fy_r,c_r z^{i_r})\\
&=p_1^{-m_1(r-1)}\sum_{\fy_1, \ldots ,\fy_r\in \bZ/p_1^{m_1}\bZ}\Bigg((N')^{-(r-1)} \prod_{1\leq j\leq r} \ \sum_{\fz_j\in \bZ/N'\bZ} f_j(\fy_j,\fz_j)\\
&\qquad \qquad \qquad \qquad \qquad \qquad +O_{i_r,d}\Big(N'\sum_{2\leq j\leq t}p_j^{-1/i_r}\log^{dr}p_j^{m_j}\Big)\Bigg).
\end{split}
\end{equation}
Summing over $\fy_1,\ldots,\fy_r$ and multiplying by $p_1^{-m_1(r-1)}$, the error term in \eqref{eq: equidistribution 3} is $O_{i_r,d}(N\sum_{2\leq j\leq t}p_j^{-1/i_r}\log^{dr}p_j^{m_j})$, while the main term is
\begin{align*}
&N^{-(r-1)}
\sum_{\fy_1,...,\fy_r
\in \bZ/p_1^{m_1}\bZ}  \ \sum_{\fz_1,\ldots,\fz_r \in \bZ/N'\bZ} f_1(\fy_1,\fz_1) \cdots f_r(\fy_r,\fz_r)\\
&=N^{-(r-1)}|B_1|\cdots|B_r|. \qedhere
    \end{align*}
\end{proof}

Since wrappers are disjoint unions of Bohr sets, we can write the characteristic functions of wrappers as sums of characteristic functions of Bohr sets to obtain the following corollary. 

\begin{corollary}
\label{equidistribution wrappers}
Let $i_1< \cdots <i_r$ and $D$ be positive integers. Let $$N=p_1^{m_1} \cdots p_t^{m_t} \quad \text{with} \quad p_1,\ldots,p_t>i_r,$$
and let $W_j$ be a $(\tau_j,d_j)$-wrapper in $\bZ/N\bZ$, with $\tau_j\geq D^{-1}$  and $d_j\leq D$ for $1\leq j\leq r$. Let $c_1, \ldots, c_r \in (\bZ/N\bZ)^{\times}$. Then
\begin{align*}
\sum_{x\in \bZ/N\bZ} \1_{W_1}(c_1 x^{i_1}) \cdots \1_{W_r}(c_r x^{i_r}) &= |W_1| \cdots |W_r| N^{-(r-1)} \\ &\qquad +O_{i_r,D}\left(N\sum_{1\leq j\leq t}\frac{\log^{Dr}p_j^{m_j}}{p_j^{1/i_r}}\right).
\end{align*}
\end{corollary}

\subsection{Dense Fermat over finite, cyclic rings}

We now establish Theorem~\ref{thm2}. 
Recall that $\eps$ is sufficiently small depending on $s,i_r,\kappa,\Omega$, and that $R$ is sufficiently large depending on $s,i_r,\kappa,\Omega$. Over the course of the proof, we will explain how small $\eps$ has to be and how large $R$ has to be, albeit implicitly.
We will work in $\bZ/N\bZ$, where 
$N=p_1^{m_1} \cdots p_t^{m_t}$ with \eqref{rough}. 
Without loss of generality, we may assume that $A\subseteq(\bZ/N\bZ)^{\times}$, at the cost of changing \eqref{density assumption mod N} to
\begin{equation}
\label{eq: dense fermat 0}
|A|\geq\left(\prod_{1\leq j\leq r}\frac{1}{rk_j}+\frac{9}{10}\kappa\right)N,
\end{equation}
since there are at most $O_{\Omega}(N/R)$ many elements that are not invertible. 

Assume for a contradiction that there are fewer than $\eps N^{s-1}$ many solutions $\bx \in A^s$ to 
\eqref{MainEq}.
Define 
\[
A_n = \{ c_n a^{i_j} : a \in A\}
\]
for all $n \in I_j$ and all $1 \leq j \leq r$. Note that $|A_n| \geq |A|/i_r^{\Omega} \gg_{i_r, s,\Omega} N$ since $A\subseteq(\bZ/N\bZ)^{\times}$. We may now apply Theorem \ref{sem} to find --- for every $1 \leq i \leq s$ 
--- a $(\tau_i, d_i)$-wrapper $W_i$ with $\tau_i \gg_{\eps, s, i_r, \Omega} 1$ and $d_i \ll_{\eps, s, i_r, \Omega} 1$ such that 
\[
|A_i \setminus W_i| \ll \eps^{1/(2s)}N
\]
and 
\begin{equation} \label{eq: dense fermat 1}
\sum_{w_1 \in W_1, \dots, w_s \in W_s} \1_{w_1 + \dots + w_s = b} \ll_{s, i_r, \Omega} \eps^{1/2} N^{s-1},
\end{equation}
for some $b \in \bZ/N\bZ$. 

We now claim that 
\begin{equation} \label{eq: dense fermat 2}
|W_1| + \dots + |W_s| < \left(1 + \frac{\kappa}{10}\right)N. 
\end{equation}
We prove this claim via contradiction, so let us assume that it does not hold. Let $C$ be a large, positive constant. As $\eps$ is sufficiently small, we have 
\[ 
|W_i| \geq |A_i| - C\eps^{1/(2s)} N \gg_{i_r,s,\Omega} N  
\]
for all $1 \leq i \leq s$. We may now apply Lemma \ref{PCD}, giving
\[ 
\sum_{w_1 \in W_1, \dots, w_s \in W_s} \1_{w_1 + \dots + w_s} \gg_{s, i_r, \kappa, \Omega} N^{s-1}. 
\]
As $\eps$ is sufficiently small, this contradicts \eqref{eq: dense fermat 1}. 

We may therefore assume that \eqref{eq: dense fermat 2} holds. For every $\bn \in I_1 \times \dots \times I_r$, let 
\[ 
X_{\bn} =\{ x \in \bZ/N\bZ : c_{n_j} x^{i_j} \in W_{n_j} \;(1 \leq j \leq r) \} . 
\]
By Corollary \ref{equidistribution wrappers},
\begin{align}
\notag
|X_{\bn}| &=|W_{n_1}| \cdots |W_{n_r}| N^{-(r-1)} \\ & \qquad +O_{\eps, s,i_r,\Omega}\left(N\sum_{1\leq j\leq t}p_j^{-1/i_r}\log^{O_{\eps,s,i_r,\Omega}(1)}p_j\right).
\label{eq: dense fermat 3}
\end{align}
Since $|A_i \setminus W_i| \ll \eps^{1/(2s)} N$ for all $1 \leq i \leq s$,
\[ |A| \leq |X_{\bn}|  + O_{s, i_r, \Omega}(\eps^{1/(2s)} N) .\]
Coupling this with the preceding bound, and then averaging over all $\bn \in I_1 \times \dots \times I_r$, gives
\begin{align} \label{eq: dense fermat 4}
\frac{|A|}{N} & \leq \frac{1}{k_1 \cdots k_r} \sum_{\bn \in I_1 \times \dots\times I_r} \frac{|W_{n_1}| \cdots |W_{n_r}| }{N^{r}} \nonumber \\
& \quad + O_{s, i_r, \Omega}(\eps^{1/(2s)}) + O_{\eps, s,i_r,\Omega}\left(\sum_{1\leq j\leq t}p_j^{-1/i_r}\log^{O_{\eps,s,i_r,\Omega}(1)}p_j\right). 
\end{align}

Clearly, we may assume that $\kap$ is sufficiently small in terms of $r$. Also, like in the previous section, we may assume that $\eps$ is a function of $s, i_r, \kap, \Omega$, so that the dependence of the implied constants on $\eps$ can be subsumed into their dependence on $s, i_r, \kap, \Omega$.

By the AM--GM inequality and \eqref{eq: dense fermat 2}, the first term on the right-hand side of \eqref{eq: dense fermat 4}
is bounded above by
\begin{align*}
\frac{1}{ N^r  k_1 \cdots k_r} \left(  r^{-1} \sum_{j=1}^r \sum_{n_j \in I_j} |W_{n_j}| \right)^r & \leq \frac{(1 + \kappa/10)^r}{k_1 \cdots k_r r^r} \\
& \leq \frac{1 +r\kappa/5}{k_1 \cdots k_r r^r} \le  \prod_{1 \le j \le r} \frac1{rk_j} +  \frac{\kappa}{5}.
\end{align*}
Now choose $\eps$ small enough to ensure that the first error term on the right-hand side of \eqref{eq: dense fermat 4} is less than $\kappa/5$, and then choose $R$ large enough to ensure that the second error term on the right-hand side of \eqref{eq: dense fermat 4} is less than $\kappa/5$. Thus,
\[
|A| < \left( \prod_{1 \le j \le r} \frac1{rk_j} +  \frac{3\kappa}{5} \right)N,
\]
contradicting the hypothesis \eqref{eq: dense fermat 0}. Hence, we must have at least $\eps N^{s-1}$ many solutions to \eqref{MainEq} with $x_1, \dots, x_s \in A$.

\subsection*{Funding}
AM is supported by Leverhulme Early Career Fellowship \texttt{ECF-2025-148}. 

\subsection*{Acknowledgements}
ZLL would like to thank his advisor Terence Tao for his guidance and support.

\subsection*{Rights}

For the purpose of open access, 
the authors have applied a Creative Commons Attribution 
(CC-BY) licence to any 
Author Accepted Manuscript version arising 
from this submission.



\begin{thebibliography}{99}


\bibitem{BW2023}
J. Br\"{u}dern and T. D. Wooley, \emph{On Waring's problem for larger powers}, J. Reine Angew. Math. \textbf{805} (2023), 115--142.

\bibitem{CGLLW}
J. Champagne, Z. Ge, T. H. L\^{e}, Y.-R. Liu and T. D. Wooley, \emph{Equidistribution of polynomial sequences in function fields: resolution of a conjecture}, arXiv:2512.16118.

\bibitem{CC2025} S. Chow and J. Chapman, \emph{Generalised Rado and Roth criteria}, Ann. Sc. Norm. Super. Pisa Cl. Sci. (5) \textbf{XXVI} (2025), 1263--1312.

\bibitem{CLP2021}
S. Chow, S. Lindqvist and S. Prendiville, \emph{Rado's criterion over squares and higher powers}, J. Eur. Math. Soc. \textbf{23} (2021), 1925--1997.

\bibitem{CLS2013}
E. Croot, I. \L aba and O. Sisask, \emph{Arithmetic progressions in sumsets and $L^p$-almost-periodicity}, Combin. Probab. Comput. \textbf{22} (2013), 351--365.


\bibitem{CGS2012}
P. Csikv\'{a}ri, K. Gyarmati, A. S\'{a}rk\"{o}zy, \emph{Density and Ramsey type results on algebraic equations with restricted solution sets}, Combinatorica \textbf{32} (2012),  425--449. 

\bibitem{Deligne}
P. Deligne, \emph{La conjecture de Weil: I}, Publ. Math. IHES \textbf{43} (1974), 273--307.

\bibitem{FGR1988}
P. Frankl, R. L. Graham and V. R\"{o}dl, \emph{Quantitative theorems for regular systems of equations}, J. Combin. Theory Ser. A \textbf{47} (1988), 246--261.

\bibitem{Gr2005}
B. Green, \emph{Roth's theorem in the primes}, Ann. of Math. (2) \textbf{161} (2005), no. 3, 1609--1636.

\bibitem{GR2005}
B. Green and I. Z. Ruzsa, \emph{Sum-free sets in abelian groups}, Israel J. Math. \textbf{147} (2005), 157--188.

\bibitem{Katz}
N. M. Katz, \emph{Estimates for mixed character sums}, Geom. Funct. Anal.
\textbf{18} (2008), 1251--1269.

\bibitem{LLW2025}
T. H. L\^{e}, Y.-R. Liu and T. D. Wooley, \emph{Equidistribution of polynomial sequences in
function fields, with applications}, Adv. Math. \textbf{479} (2025), 110424.

\bibitem{LP2010}
H. Li and H. Pan, \emph{A density version of Vinogradov's three primes theorem}, Forum Math. \textbf{22} (2010), 699--714.

\bibitem{Lim2025}
Z. L. Lim, \emph{On the density version of quadratic Waring's problem and the quadratic Waring--Goldbach problem}, preprint 2025, arXiv:2508.14939.

\bibitem{Lin2018}
S. Lindqvist, \emph{Partition Regularity of Generalised Fermat Equations}, Combinatorica \textbf{38} (2018), 1457--1483.

\bibitem{Pre2021}
S. Prendiville, \emph{Counting Monochromatic Solutions to Diagonal Diophantine Equations}, Discrete Analysis (2021), doi.org/10.19086/da.28173.

\bibitem{Sal2021}
J. Salmensuu, \emph{A density version of Waring's problem}, Acta Arith. \textbf{199} (2021), 383--412.

\bibitem{Sem2025}
A. Semchankau, \emph{A new bound for $A(A+A)$ for large sets}, J. Number Theory \textbf{268} (2025), 142--162.

\bibitem{Sha2014}
X. Shao, \emph{A density version of the Vinogradov three primes theorem}, Duke Math. J. \textbf{163} (2014), 489--512.

\bibitem{TV2006}
T. Tao and V. H. Vu, \emph{Additive combinatorics},
Cambridge Stud. Adv. Math. \textbf{105,}
Cambridge University Press, Cambridge, 2010. xviii+512 pp.

\bibitem{Va1959}
P. Varnavides, \emph{On certain sets of positive density}, J. London Math. Soc. \textbf{34} (1959), 358--360.

\bibitem{Vau1997}
R. C. Vaughan, \emph{The Hardy--Littlewood method}, 2nd edition, Cambridge University
Press, Cambridge, 1997.

\end{thebibliography}
\end{document}